\documentclass[a4paper,12pt,leqno]{article}
\usepackage{geometry} 
\usepackage{bussproofs}
\usepackage{amssymb,amsmath}
\usepackage{proof}
\usepackage[utf8]{inputenc}
\usepackage{graphicx}
\graphicspath{ {images/} }
\usepackage{qtree}
\usepackage{mathtools}
\usepackage{hyperref}
\usepackage{listings}
\usepackage{indentfirst}
\usepackage{titlesec}
\numberwithin{equation}{section}
\newtheorem{defn}[equation]{Definition}

\newtheorem{rem}[equation]{Remark}
\newtheorem{exm}[equation]{Example}
\newtheorem{lemma}[equation]{Lemma}

\newtheorem{notat}[equation]{Notation}
\newtheorem{newpar}[equation]{}
\newtheorem{xdefn}{Definition.}
\newtheorem{xproposition}{Proposition.}
\newtheorem{xcorollary}{Corollary.}
\newtheorem{xrem}{Remark.}
\newtheorem{xexm}{Example.}
\newtheorem{xlemma}{Lemma.}
\newtheorem{xtheorem}{Theorem.}
\newtheorem{xnotat}{Notation.}
\newtheorem{xnewpar}{\it}
\newtheorem{xproof}{{\it Proof. }}
\newtheorem{xproofof}{{\it Proof}}
\def\qed{\hspace{0.3cm}{\rule{1ex}{2ex}}}

\newenvironment{definition}{\begin{defn}\em}{\end{defn}}

\newenvironment{example}{\begin{exm}\em}{\end{exm}}

\newenvironment{proof}{\begin{xproof}\em}{\end{xproof}}

\newenvironment{newparagraph*}[1]{\begin{xnewpar}\hspace*{-1.5mm}{#1}. \rm}{\end{xnewpar}}
\newenvironment{definition*}{\begin{xdefn}\em}{\end{xdefn}}
\newenvironment{remark*}{\begin{xrem}\em}{\end{xrem}}
\newenvironment{example*}{\begin{xexm}\em}{\end{xexm}}
\newenvironment{notation*}{\begin{xnotat}\em}{\end{xnotat}}
\newenvironment{proposition*}{\begin{xproposition}}{\end{xproposition}}
\newenvironment{corollary*}{\begin{xcorollary}}{\end{xcorollary}}
\newenvironment{lemma*}{\begin{xlemma}}{\end{xlemma}}
\newenvironment{theorem*}{\begin{xtheorem}}{\end{xtheorem}}

\begin{document}
	
	\title{On the Soundness of Bealer's Logics}
	\author{Clarence Protin}
	
	\date{August 2021}
	
	\maketitle
	
	\begin{abstract} 
	
	Bealer's intensional logics T1 and T2 were proposed and expounded most fully in his book \emph{Quality and Concept} (1982) \cite{QC} as well in \cite{C}. These logics are unique in being  extensions of classical first-order	associated to a non-nominalist or non-inscriptionalist ontology and theory of meaning.  Structurally they are similar to the second-order systems proposed about the same time by Zalta \cite{zalta}. In the book and article referenced above Bealer presents a detailed sketch of a proof of soundness and completeness for T1 and T2 (something which seems to be lacking for Zalta's systems). However there are key steps to the soundness proofs which are stated without proof and which seem to be non-trivial.  In this paper we both simplify the original presentation of systems T1 and T2 and supply the rather complex and involved proofs of Bealer's missing lemmas.

	\end{abstract}


This paper only addresses technical details of Bealer's logic. For a recent survey of the intensional logic from a philosophical perspective which includes a discussion of the merits of Bealer's approach see \cite{parsons}.
The aims of the paper are (i) to give an alternative simplified presentation of Bealer's syntactic decomposition and (ii) to provide full details of the non-trivial lemmas required to the prove the soundness of T1 and T2.
	
The paper is organised as follows.
In section 1 we present a slightly different and hopefully simpler version of the common framework for T1 and T2. In sections 2 and 3 we present systems T1 and T2 together with same basic results. In section 4 we introduce the class of models we will be working with. In section 5 we present the missing proofs of all the lemmas required to prove the soundness of T1 and in section 6 we do the same for T2.

\section{Bealer Decomposition}

\begin{definition} The language $L^\omega$ consists of a countable ordered collection
	of variables $x,y,z,...$ and $n$-ary predicate symbols ($n\geq 1$) $F,G,H,...$
	with a distinguished binary predicate $=$, logical connectives $\&,\sim,\exists$ and the intensional abstraction
	operator $[\quad]_{x_1...x_n}$ where $x_1...x_n$ is a possibly empty sequence of distinct variables.
\end{definition}

We use the notation $\bar{x}$ for a (possibly empty) sequence of distinct variables. We use lower-case letters (possibly with subscripts) towards the end of the Latin alphabet to denote variables.

\begin{definition} Formulas and terms of $L^\omega$ are defined by simultaneous induction:
	
	\begin{itemize}
		\item Variables are terms.
		\item If $t_1,...,t_n$ are terms and $F$ is $n$-ary predicate, then $F(t_1,...,t_n)$ is a formula.
		\item If $A$ and $B$ are formulas, $v$ a variable, then $(A\& B), \neg A$ and $\exists v . A$ are formulas.
		\item If $A$ is a formula and $v_1,...,v_m$, $0\leq m$, is a sequence of \emph{distinct} variables, then $[A]_{v_1...v_m}$
		is a term.
		
	\end{itemize}	

\end{definition}
A term of the form $[A]_{v_1,...,v_n}$ is called an \emph{intensional abstract}. Intensional abstracts of the form $[F(v_1,...,v_n)]_{v_1...v_n}$ are called \emph{elementary}. Intensional abstracts $[A]_{\alpha}$ where $A$ is not an atomic formula are called \emph{complex}.

The intensional abstraction operator generalises the set-theoretic extension operator $\{ x: A\}$. The standard notions of bound and free variable,  bound-variable renaming ($\alpha$-equivalence) as well as a term being free for $x$ in $A$ carry over to $L^{\omega}$ in the expected way (note that intensional abstraction binds variables like quantifiers).  Terms that differ only by bound variable renaming will be called \emph{variants}.

We will define a set of (partially defined) syntactic operations on intensional abstracts and an algorithm to decompose an intensional abstract in a unique way in terms of such syntactic operations, elementary intensional abstracts and variables.

A term $[A]_{\bar{x}}$ is \emph{normalized} if all the variables in $\bar{x}$
	are free in $A$ and  they display the order in which these variables first occur free.

  If $A$ is atomic then we call $[A]_{\bar{x}}$ a	\emph{prime term} if the variables in $\bar{x}$ are free in $A$, independently of the order.

	 Given an atomic formula $F(t_1,...,t_n)$ if a variable occurs free in more than one of the terms $t_i$ then it is called
	a \emph{reflected variable}. A prime term $[A]_{\bar{x}}$ is called a \emph{prime reflection term} if $\bar{x}$ 
	contains a reflected variable in $A$.

	Consider a prime term $[F(t_1,..., [B]_{\bar{y}},...,t_j)]_{\bar{x}}$ which is not a prime reflection term.
	If there is a variable in $\bar{x}$ which is free in $[B]_{\bar{y}}$
	and all the previous arguments are variables in $\bar{x}$, 
	then the prime term is a \emph{prime relativized predication term} and such a variable
	is called a \emph{relativized variable}.

	 Consider a prime term $[F(t_1,...,t_k,...,t_j)]_{\bar{x}}$ which is not a prime reflection term.
	If there is term $t_k$ which is a variable not in $\bar{x}$ or
	which contains no free variables which are in $\bar{x}$ and all the previous
	arguments are variables in $\bar{x}$, then it is a \emph{prime absolute predication term}.

	These definitions allow us to divide all intensional abstracts into seven disjoint categories. The following is clear:
	
	\begin{lemma}
	
	Given an intensional abstract $[A]_{\bar{x}}$ it falls into one and one only of the 
	following seven categories which depend only on $A$ and $\bar{x}$ independently of order (only on the underlying set). They are also invariant under the renaming of bound variables.
	
	\begin{enumerate}
		\item Complex and $\bar{x}$ contains variables not free in $A$.
		\item Complex and all the variables in $\bar{x}$ are free in $A$. 
		\item Not complex and  $\bar{x}$ contains variables not free in $A$.
		\item Prime reflection term.
		\item Prime relativized predication term.
		\item Prime absolute predication term.
		\item $[A]_{\bar{x}'}$ is elementary for some permutation $\bar{x}'$ of $\bar{x}$.
	\end{enumerate}
	\end{lemma}

Before defining Bealer decomposition we need the following considerations on permutations. It is is clear that any permutation $\sigma^n \in S^n$($n\geq 2)$ can be decomposed into
permutations of the form $\sigma^n_c = (1\quad 2\quad ...\quad n)$(for $n\geq 3$) and $\sigma^n_i = (n-1\quad n)$. However, we need a uniquely defined decomposition defined as follows.
Assume that our permutations are acting on sequences $x_1...x_n$ of variables and let
$x'_1...x'_n$ be a permutation of $x_1...x_n$.  We use the notation $\sigma^n x_1....x_n$ to denote the results of apply $\sigma^n$ to the sequence. We associate to $x_1...x_n$ and $x'_1...x'_n$ (or equivalently the permutation which transforms the first sequence into the second) a uniquely defined decomposition as follows:

Suppose $\pi' = \pi_n...\pi_2\pi_1$ has already been defined, where $\pi_i$ is either $\sigma^n_c$ or $\sigma^n_i$.  Then let $s$ be the maximal subsegment of $\pi'x_1...x_n$ equal to an initial segment $b =x'_1...x'_k$ of $x'_1...x'_n$. If $b$ is the whole $x'_1...x'_n$ we are done. Otherwise $x'_{k+1}$ will correspond
to a certain $x$ in $\pi' x_1...x_n$. If $x$ is the last element of this sequence then we choose $\sigma^n_i$, otherwise we chose $\sigma^n_c$. This process terminates and the decomposition is clearly unique. Given a permutation $\sigma$ (or two sequences $\bar{x},\bar{x'}$ of the same length),  we denote its decomposition by $\beta(\sigma)$(or $\beta(\bar{x},\bar{x'}))$.

\begin{example} Consider $x_1x_2x_3x_4$ and $x_2 x_4 x_3 x_1$.
The the process yields $\sigma^4_i$	and $x_1x_2x_4x_3$, $\sigma^4_c\sigma^4_i$ and
$x_3x_1x_2x_4$,  $\sigma^4_c\sigma^4_c\sigma^4_i$ and
$x_4x_3x_1x_2$ and finally  $\sigma^4_c\sigma^4_c\sigma^4_c\sigma^4_i$ and $x_2x_4x_3x_1$.
\end{example}

\begin{definition}
The Bealer operations consists in eight syntactic operations defined on intensional abstracts. These consists in six unary operations $C,I, N,U,E, R$, a binary operation $K$ and partially-defined binary operations $P_n$ for $n\geq 0$. These operations can be seen as defined on equivalence classes modulo renaming of bound variables. 
Let $[A]_{\bar{x}}$ be an intensional abstract with $\bar{x}$ of length $n$. Then we have:

\[C[A]_{\bar{x}} = [A]_{\sigma^n_c\bar{x}} \quad n\geq 3 \]

\[I[A]_{\bar{x}} = [A]_{\sigma^n_i\bar{x}} \quad n\geq 2 \]

\[N[A]_{\bar{x}} = [\neg A]_{\bar{x}} \]

\[U[A]_{\bar{x}} = [\exists x_n A]_{x_1...x_{n-1}}\quad n\geq 1 \]

\[E[A]_{\bar{x}} = [A]_{\bar{x}y} \quad \text{ $y$ the first variable not in $A$ or $\bar{x}$}\]

\[R[A]_{\bar{x}} = [A[x_{n-1} / x_n]]_{x_1....x_{n-1}} \]

\[K[A]_{\bar{x}}[B]_{\bar{y}}  = [A\&B']_{\bar{x}} \quad \text{$[B']_{\bar{x}}$ variant of $[B]_{\bar{y}}$ } \]

\[P_0[A]_{\bar{x}}[B]_{\bar{y}}  = [A[ [B]_{\bar{y}}/x_n] ]_{x_1...x_{n-1}} \quad \text{where $n\geq 1$ and  $[B]_{\bar{y}}$  is free for $x_n$ in $[A]_{x_1...x_{n-1}}$} \]
	
\[P_n[A]_{\bar{x}}[B]_{\bar{y}} =  [A[ [B]_{\bar{z}}/x_n] ]_{x_1...x_{n-1}\bar{w} }\quad
\text{where $\bar{y} = \bar{z}\bar{w}$,$\bar{w}$ length $n\geq 1$,$[B]_{\bar{z}}$  is free for $x_n$ in $[A]_{x_1...x_{n-1}}$ }  \]	

\end{definition}

Consider a sequence $\bar{x}$ and a permutation $\bar{x'}$. We associate
to $\beta(\bar{x},\bar{x'})$ a sequence $\Sigma(\bar{x},\bar{x'})$ of operations $I$ and $C$ obtained by replacing $\sigma^n_i$ with $I$ and $\sigma^n_c$ with $C$ in $\beta(\bar{x},\bar{x'})$. Thus in the previous example we have $\Sigma(x_1x_2x_3x_4,x_2 x_4 x_3 x_1) = CCCI$. This sequence is to be interpreted compositionally. In general we shall use the notation $\Sigma$ for any  (possibly empty) sequence of operations $C$ and $I$.
We use the notation $E^n$ to denote $E$ composed $n$ times. Also we shall use the notation $P^n_0tt_1...t_n$ to mean $P_0... P_0(P_0tt_1)t_2....t_n$.

	We now define Bealer decomposition.
	
	\begin{definition}
Given an intensional abstract $[A]_{\bar{x}}$ or variable $x$ the Bealer decomposition $\mathcal{B} [A]_{\bar{x}}$	is defined inductively as follows:

\begin{itemize}
	\item Suppose $[A]_{\bar{x}}$ belongs to categories 1 or 3. And let $\bar{x}'$ 
	be a permutation of $\bar{x}$ of the form $\bar{y}\bar{z}$ so that  $[A]_{\bar{y}}$ is normalized and $\bar{z}$ (of length $m$)
	consists of the variables in $\bar{x}$ that do not occur free in $A$ in order of their occurrence in 
	$\bar{x}$. Then  $\mathcal{B}[A]_{\bar{x}}:= \Sigma(\bar{x},\bar{x'})E^m\mathcal{B}[A]_{\bar{y}}$.

	\item Suppose  $[A]_{\bar{x}}$ belongs to category 2 and is not normalized. And let $\bar{x}'$ be a permutation of $\bar{x}$ 
	so that  $[A]_{\bar{x}'}$ is normalized. Then $\mathcal{B}[A]_{\bar{x}}:=\Sigma(\bar{x},\bar{x'})\mathcal{B}[A]_{\bar{x}'}$.

	\item Suppose the term belongs to category 2 and is normalized. Then we have

	$\mathcal{B}[A \& B]_{\bar{x}}:= K\mathcal{B}[A]_{\bar{x}} \mathcal{B}[B]_{\bar{x}}$
	
	$\mathcal{B}[\neg A]_{\bar{x}}:= N\mathcal{B}[A]_{\bar{x}}$
	
	$\mathcal{B}[\exists v A]_{\bar{x}}:= U\mathcal{B}[A]_{\bar{x} v}$  where $v$ is the first variable not in $A$ or $\bar{x}$.

	\item Let the term belong to category 4. Of the reflected variables in $\bar{x}$ consider the one $v$ with
	the right-most occurrence and let be $t$ the right-most argument of $A$ in which $v$ has a free occurrence. Let $w$ be the
	alphabetically earliest variable not occurring in $A$ or $\bar{x}$.  Let $\bar{x}'$ be the permutation of
	$\bar{x}$ such that $[F(t_1,...,t_n)]_{\bar{x}'}$ is normalized and consider $\bar{y} v$ where $\bar{y}$ is $\bar{x}'$ with $v$ removed.
	Then 
	\[\mathcal{B}[F(t_1,...,t_n)]_{\bar{x}}:=\Sigma(\bar{x},\bar{y} v)R\mathcal{B}[F(...t[w/v]...)]_{\bar{y} v w}\]
	
	\item Let the term belong to category 6 and let $t$ be the argument which is not a variable in $\bar{x}$ or
	which contains no free variables which are in $\bar{x}$ and such that all the previous
	arguments are variables in $\bar{x}$. Let $v$ be the alphabetically earliest variable
	not occurring in $A$ or $\bar{x}$.. Let $\bar{x}'$ be the permutation of $\bar{x}$ such that $[F(t_1,...,t,...,t_n)]_{\bar{x}'}$ 
	is normalized. Then
	\[\mathcal{B}[F(...,t,...)]_{\bar{x}}:=\Sigma(\bar{x},\bar{x'})P_0\mathcal{B}[F(...,v,...)]_{\bar{x}' v} \mathcal{B}t\]

	\item Let the term belong to category 5. Then it is  a prime term $[F(t_1,..., [B]_{\bar{y}},...,t_j)]_{\bar{x}}$, 
	where there is a variable in $\bar{x}$ which is free in $[B]_{\bar{y}}$
	and such that all the previous arguments are variables in $\bar{x}$. 
	Let $\bar{z}$ be the sequence, as they first occur, of all $m$ free variables in  $[B]_{\bar{y}}$ which are in $\bar{x}$.
	Let $v$ be the alphabetically earliest variable not occurring in $A$ or $\bar{x}$.
	Then if $\bar{x}'$ is the normalizing permutation of $\bar{x}$ it is of the form $\bar{x}''\bar{z}\bar{x}'''$. Let $\bar{w} = \bar{x}''\bar{x}'''\bar{z}$. Then
	\[\mathcal{B}[F(...,[B]_{\bar{y}},...)]_{\bar{x}}:=\Sigma(\bar{x}, \bar{w})P_m\mathcal{B}[F(...,v,...)]_{\bar{x}''\bar{x}'''v}, \mathcal{B}[B]_{\bar{y}\bar{z}}\]
	
	\item Finally let the term be in category 7 and let $\bar{x}'$ be such that $[F(x_1,...,x_n)]_{\bar{x}'}$ is elementary.
	\[\mathcal{B}[F(x_1,...,x_n)]_{\bar{x}}:=\Sigma(\bar{x},\bar{x}')[F(x_1,...,x_n)]_{\bar{x}'}\]
	
	\item In the case of a variable we have $\mathcal{B}x := x$.
\end{itemize}
	\end{definition}
It is easy to check that this decomposition is well-defined and unique modulo renaming of bound variables. Bealer decomposition transforms an intensional abstract into an expression (in Polish notation) involving the eight Bealer operations, elementary abstracts and variables.

\begin{example}
	Consider the term $[\neg F(x, [G(x,y)]) ]_y$. Then 
	
	\[\mathcal{B}[\neg F(x, [G(x,y)]) ]_y= 
	N\mathcal{B}[F(x, [G(x,y)]) ]_y = NP_0\mathcal{B}[F(v, [G(x,y)]  )]_{yv} x\]
	\[= NP_0I\mathcal{B}[F(v, [G(x,y)]  )]_{vy} x = NP_0IP_1\mathcal{B}[F(v, w  )]_{vw}\mathcal{B}[G(x,y)]_y  x \] 
	\[=NP_0IP_1\mathcal{B}[F(v, w  )]_{vw}P_0\mathcal{B}[G(u,y)]_{yu} x  x= NP_0IP_1\mathcal{B}[F(v, w  )]_{vw}P_0I\mathcal{B}[G(u,y)]_{uy} x  x\]
\end{example}

The following observation will be used further ahead:

\begin{lemma} Let $[F(t_1,...,t_n)]_{\bar{x}}$ be  a prime term. Then its Bealer decomposition has the form
	
	\[ \Gamma_1...\Gamma_m[F(v_1,...,v_n)]_{v_1...v_n}\mathcal{B}t'_{i_1}....\mathcal{B}t'_{i_m}\]
	
	where the $\Gamma_i$ consist in sequences of the form
	
	\[\Sigma_1R\Sigma_2R\Sigma_3...\Sigma_{l-1}R\Sigma_lP_{n_i} \]
	
	and $t'_{i_k} = s_{\bar{y}}$ where  $s = t_{i_k}$ and if $s = [A]_{\bar{z}}$ then
	$s_{\bar{y}}$ is $[A]_{\bar{z}\bar{y}}$ where $\bar{y}$ is the sequence of free variables
		(as they first occur) of $s$ which are in $\bar{x}$. If $s$ is a variable then $\bar{y}$ is empty and $t'_{i_k} = s$.
		 
	\end{lemma}

\section{System T1}

We now consider Bealer's axiomatic system T1 over the language $L^{\omega}$.
The conception of intensionality behind T1 is that there is only one necessary truth
and that the intensional abstracts of necessarily equivalent formulas are equal
and interchangeable.  We define necessitation in terms of equality of intensional abstracts:
 \[\square A \equiv_{df} [A] = [[A] = [A]]\].
 
 Here $[[A] = [A]]$ represents the one necessary truth.
 
A Hilbert system for T1 consists of the following axioms and rules:
\begin{enumerate}
	\item All propositional tautologies 
	\item (Inst) $\forall v. A(v) \rightarrow A(t)$, $t$ free for $v$ in A.
	\item (QImp) $\forall x. (A\rightarrow B) \rightarrow (A \rightarrow \forall x. B)$, $x$ not free in $A$
	\item (Id) $x = x$
	\item (L) $x = y \rightarrow (A(x,x) \leftrightarrow A(x,y))$,  $y$ free for $x$ in the positions it replaces $x$.
	\item $[A]_{\bar{x}} \neq [B]_{\bar{y}}$, for $\bar{x}$ and $\bar{y}$ of different lengths
	\item $[A]_{\bar{x}} = [A']_{\bar{x}'}$ (equality modulo renaming bound variables)
	\item (B) $[A]_{\bar{x}} = [B]_{\bar{x}} \leftrightarrow \square \forall \bar{x}. (A \leftrightarrow B)$
	\item (T) $\square A \rightarrow A$
	\item (K) $\square (A \rightarrow B) \rightarrow (\square A \rightarrow \square B)$
	\item (S5) $\lozenge A \rightarrow \square \lozenge A$
	
\end{enumerate}
\begin{enumerate}
	
	\item (MP) From $A$ and $A\rightarrow B$ infer $B$ 
	
	\item (N) From $A$ infer $\square A$ 
	
	\item (Gen) From $A$ infer $\forall x. A$ 
	
\end{enumerate}

Here, as usual, $\lozenge A =_{df} \neg \square \neg A$. 

In order to prove soundness and completeness it is convenient to replace
axioms (L),(B)  and (S5) with:

\begin{enumerate}

\item (L') $x = y \rightarrow (A(x,x) \rightarrow A(x,y))$,  $y$ free for $x$ in the positions it replaces $x$, $A$ atomic.
\item (B1) $\square(A\leftrightarrow B)\leftrightarrow [A]=[B]$
\item (B2) $\forall v.[A(v)]_{\bar{x}} = [B(v)]_{\bar{x}}\leftrightarrow [A(x)]_{\bar{x}v} = [B(v)]_{\bar{x}v}$

\item (S5') $ x \neq y \rightarrow \square  x \neq y$
\end{enumerate}

The rest of this section is devoted to proving that we obtain equivalent axiomatic systems.

\begin{lemma}
	In T1 we have:
	
	i) $\vdash x = y \rightarrow \square x = y $
	
	ii) $\vdash \neg\square x = y \leftrightarrow  x \neq y$
	
\end{lemma}

\begin{proof}
	For i), by the identity axiom we have $x=x$. Hence using the necessitation rule we get
	
	\[ \square x = x \]
	
	which is
	
	\[[x = x] = [ [x = x] = [x = x] ]  \]
	
 We denote this expression by $B(x,x)$. Using $B(x,x)$ in (L)
	
	\[ x = y \rightarrow A(x,x) \leftrightarrow A(x,y) \]
	
	which is equivalent to $A(x,x) \rightarrow ( x = y \rightarrow A(x,y))$ we obtain as desired.\\
	
	For ii), take the converse of the (T) instance $\square x = y \rightarrow x = y$ 
	to obtain $x \neq y \rightarrow \neg \square x = y$.
	For the other direction simply take the converse of i)
	\qed
\end{proof}
\begin{lemma}
	
	We obtain an equivalent system if in T1 we replace (S5) with
	
	(*) $ x \neq y \rightarrow \square  x \neq y$.
	
\end{lemma}

\begin{proof}
	
	Assuming (*), since $\lozenge A$ is by definition $\neg ([\neg  A] = [ [\neg  A] = [\neg  A]])$ we get (S5) immediately.
	
	In the other direction take the (S5) instance
	
	\[ \neg \square \neg \neg x = y \rightarrow \square \neg \square \neg \neg x = y \]
	
	that is,
	
	\[ \neg\square x = y \rightarrow \square \neg \square x = y \tag{**} \]
	
	Now by ii) of the previous lemma we have
	
	\[\neg \square x = y \rightarrow \neg x = y \]
	
	Applying necessitation and using (K) we get
	
	\[\square \neg \square x = y \rightarrow \square \neg x = y \]
	
	and combining with (**) we get as desired.
	\qed
\end{proof}

We will use the standard properties of S5 modal logic:

\begin{lemma} In T1 we have
	
	i) $\vdash A\rightarrow\lozenge A$
	
	ii)  $\vdash A \rightarrow \square \lozenge A$ 
	
	iii) $\lozenge \square A \rightarrow A$
	
	iv) If $\vdash A\rightarrow B$ then we can derive $\vdash \lozenge A \rightarrow \lozenge B$
	
	v) We have $\vdash \lozenge A \rightarrow B$ iff  $\vdash A \rightarrow \square B$ 
	
\end{lemma}

\begin{proof}
	For i) take the converse of the (T) instance $\square \neg A\rightarrow \neg A$.
	For ii) combine i) and (S5).
	For iii) take the converse of the (S5) instance $\neg \square \lozenge \neg A \rightarrow \neg \lozenge \neg A$ and use (T).
	For iv) take the converse of the premise, apply (T) and take the converse again.
	Finally v) follows easily from (T), iv), iii) and ii).
	\qed
\end{proof}

\begin{lemma}
	$\vdash \square \forall v. A(v) \rightarrow \forall v .\square A(v)$
\end{lemma}

\begin{proof}
	
	Take the (Ins) instance
	
	\[ \forall v. A(v) \rightarrow A(v) \]
	
	Applying necessitation, (K) and modus ponens yields:
	
	\[ \square \forall v. A(v) \rightarrow \square A(v) \]
	
	Applying generalization:
	
	\[ \forall v. (\square \forall v. A(v) \rightarrow \square A(v)) \]
	
	Since $v$ is not free on the left side of the implication we can apply
	(QImp) and modus ponens to obtain
	
	\[ \square \forall v. A(v) \rightarrow \forall v. \square A(v) \]
	
	\qed
\end{proof}

\begin{lemma} We have the \emph{Barcan Formula} $\vdash \forall v.\square A(v) \rightarrow \square \forall v. A(v)$.
\end{lemma}

\begin{proof}
	
	We have by (Ins)
	
	\[\forall v. \square A \rightarrow \square A \]
	
	Applying $\lozenge$ yields, by lemma 2.3 iii) and iv):
	
	\[\lozenge \forall v. \square A \rightarrow \lozenge \square A \rightarrow A \]
	
	Hence by generalization, (QImp) and modus ponens:
	
	\[\lozenge \forall v. \square A \rightarrow \forall v. A \]
	
	which again by lemma 2.3 yields
	
	\[\forall v .\square A \rightarrow \square \forall v. A \]
	
	\qed
\end{proof}

We can now prove:

\begin{lemma}
	We obtain an equivalent system  to T1 if instead of (B) take

	\[\square (A \leftrightarrow B ) \leftrightarrow ([A] = [B])\tag{B'1}\]

	\[\forall v. ( [A(v)]_{\bar{x}} =   [B(v)]_{\bar{x}}) \leftrightarrow  [A(v)]_{\bar{x} v} =  [B(v)]_{\bar{x} v}\tag{B'2}\]
	
\end{lemma}

\begin{proof}
	
	Using Barcan's formula and its converse it is easy to see that
	(B'1) and (B'2) follow from (B).
	
	Assume (B'1) and (B'2). We must show that 
	
	\[  [A]_{\bar{x}} = [B]_{\bar{x}} \leftrightarrow \square \forall \bar{x}. (A \leftrightarrow B) \]
	
	Let $\bar{x} = \bar{x'}v$.
	
	Since $[A]_{\bar{x'} v}  = [B]_{\bar{x'} v}  \leftrightarrow \forall v. ([A(v)]_{\bar{x'}} = [B(v)]_{\bar{x'}})$
	using (B'2) repeatedly we get
	
	\[[A]_{\bar{x} v}  = [B]_{\bar{x} v}  \leftrightarrow \forall \bar{x}, v. ([A(v)] = [B(v)])\]
	
	Using (B'1) we get
	\[  [A]_{\bar{x}} = [B]_{\bar{x}} \leftrightarrow \forall \bar{x}. \square (A \leftrightarrow B) \]

	Using repeatedly Barcan's formula and its converse we get as desired.
	
	\qed
	
\end{proof}

Finally:

\begin{lemma}  We obtain an equivalent system to T1  when the Leibniz axiom (L)  is restricted to atomic predicates and restricted to the  form (L') $x = y \rightarrow (A(x,x) \rightarrow A(x,y))$.
	
\end{lemma}

\begin{proof}
	L obviously implies (L') as a particular case.
	In the other direction we first show that $x = y \rightarrow (A(x,x) \leftarrow A(x,y))$.
	We have that $x = y \rightarrow (A(x,x) \rightarrow A(x,y))$ and $A(x,y)$ arises
	from $A(x,x)$ by replacing some occurrences of $x$ by $y$. So $A(x,x)$ arises
	from replacing some occurrences of $y$ in $A(x,y)$  by $x$. Hence we have $y = x \rightarrow (A(x,y) \rightarrow A(x,x))$.
	Also from (L') we can derive $x = y \rightarrow y = x$ and using this we get as desired.
	We can now proceed inductively on the structure of the formula $A$. The case of negation is immediate.
	For conjunction we use the tautology $(A\rightarrow B) \rightarrow (C \rightarrow D)\rightarrow (A\& C \rightarrow B \& D)$
	and for quantification we use generalisation, (QImp) and modus ponens.
	\qed
\end{proof}

This lemma justifies that "there is only one necessary truth".

\begin{lemma} We have $\forall x, y . [x=x] =  [y=y]$.
\end{lemma}

\begin{proof}
	
	We have $x = x \leftrightarrow y = y$. We apply necessitation and (B'1) and then generalisation to get as desired. 
	\qed
\end{proof}

\section{System T2}

Bealer's logic T2 is given the following Hilbert axiomatic system:
	
	\begin{enumerate}
		\item All propositional tautologies
		\item (Ins) $\forall x. A(x) \rightarrow A(t)$, $t$ free for $v$ in $A$.
		\item (QImp) $\forall x. (A\rightarrow B) \rightarrow (A \rightarrow \forall x. B)$, $x$ not free in $A$
		\item (Id) $x = x$
		\item (L) $x = y \rightarrow (A(x,x) \leftrightarrow A(x,y))$,  $y$ free for $x$ in the positions it replaces $x$.
		\item $[A]_{\bar{x}} \neq [B]_{\bar{y}}$, $\bar{x}$ and $\bar{y}$ of different lengths
		\item $[A]_{\bar{x}} = [A']_{\bar{x}'}$ (equality modulo renaming bound variables)
		\item $[A]_{\bar{x}} =  [B]_{\bar{x}} \rightarrow (A \leftrightarrow B)$
		
		\item $[A]_{\bar{x}} \neq [B]_{\bar{y}}$, non-elementary terms belonging to different categories.
		
		\item Let $t' = Bt$ and $r' = Br$ where $B$ is the first operation in the Bealer decomposition of $t'$ and $r'$.
		Then $t = r \leftrightarrow t' = r'$
		
		\item Let $t = Bt'r'$ and $r = Bt''r''$ where  $B$ is the first operation in the Bealer decomposition of $t$ and $r$.
		The $t = r \leftrightarrow t ' = t'' \& r' = r''$.
		
		\item If $[F(x_1,...,x_n)]_{x_1...x_n} = s$  and $G$ occurs in $s$ then $[G(y_1,...,y_m)]_{y_1...y_m} \neq t$ whenever $F$ occurs in $t$ (Non-circularity)

	\end{enumerate}
	\begin{enumerate}
		
		\item (MP) From $A$ and $A\rightarrow B$ infer $B$ 
		
		\item (TGen) Suppose that $F$ does not occur in $A(v)$.
		If $\vdash A([F(x_1,...,x_n)]_{x_1...x_n})$ then we can infer
		$\vdash  A(t)$ whenever $t$ is complex, of arity $n$ and free for $v$ in $A$.
		
		\item (Gen) From $A$ infer $\forall x. A$ 
		
	\end{enumerate}

	\section{Model Structures}

	Let $F = \emptyset$ and $T= \{\emptyset\}$. A model $\mathcal{M}$ consists of a set $\mathcal{D}$ with a decomposition into the union of \emph{disjoint} sets $\mathcal{D}_i$ for $i\geq -1$,
	a distinguished element $id \in \mathcal{D}_2$, a set $\mathcal{H}$ of functions $H$ whose domain is $\mathcal{D}$ and such that
	$H$ is the identify on $\mathcal{D}_{-1}$ and on $\mathcal{D}_0$ has range $\{T,F\}$ and on 
	$\mathcal{D}_i$ has range $\mathcal{P}(\mathcal{D}^i)$ where $\mathcal{D}^i = \mathcal{D}\times...\times \mathcal{D}$ ($i$ terms) for $i\geq 1$.
	All these functions must satisfy $H(id) = \{(x,y): x = y\}$. There is a distinguished function $\mathcal{G}$ representing the \emph{actual extension}.
	Furthermore we have unary (partial) operations $n,e,u,c,i,r$ and binary (partial)
	operations $k, p_n$ with $n\geq 0$. These operations restrict to the following domains and ranges:\\

		 $n: \mathcal{D}_i \rightarrow  \mathcal{D}_i, i \geq 0$

		 $e:\mathcal{D}_i\rightarrow  \mathcal{D}_{i+1}, i \geq 0$

	 $u:\mathcal{D}_i\rightarrow \mathcal{D}_{i-1}, i\geq 1$ and $u:\mathcal{D}_0\rightarrow  \mathcal{D}_0$.

	 $c:\mathcal{D}_i\rightarrow  \mathcal{D}_i, i \geq 3$

		  $i:\mathcal{D}_i\rightarrow  \mathcal{D}_i, i \geq 2$

		  $r:\mathcal{D}_i\rightarrow \mathcal{D}_{i-1}, i \geq 2$

		  $a: \mathcal{D}_i\times\mathcal{D}_i\rightarrow  \mathcal{D}_i, i \geq 0 $

		 $p_n:\mathcal{D}_i\times\mathcal{D}_j\rightarrow \mathcal{D}_{i+n-1}, i \geq 1, j\geq n$\\

	We define inductively $p_0^n(d, x_1,...,x_n) = p_0(p_0^{n-1}(d, x_2,...,x_{n-1}), x_1)$ and
	$p_0^1(d,x_1) = p_0(d,x_1)$.

	Furthermore we have the important constraints on how the functions in $\mathcal{H}$ relate to these operations. 
	Here $d$ denotes a suitable member of $\mathcal{D}^i$. For all $\mathcal{H}\in H$
	we must have:\\

		$d'\in \mathcal{H}n(d) \text{ iff } d' \notin \mathcal{H}d$
		
		 $(x_1,...,x_i,x_{i+1})\in \mathcal{H}e(d) \text{ iff } (x_1,...,x_i) \in \mathcal{H}d$

		$(x_1,...,x_{i-1})\in \mathcal{H}u(d) \text{ iff there is a $x_i$ such that } (x_1,...,x_i) \in \mathcal{H}d$  
		
		$(x_1,...,x_{i-1},x_i)\in \mathcal{H}c(d) \text{ iff } (x_i, x_1,...,x_{i-1}) \in \mathcal{H}d$  
		
		$(x_1,...,x_{i-1},x_i)\in \mathcal{H}i(d) \text{ iff } (x_1,...,x_i,x_{i-1}) \in \mathcal{H}d$  
		
		 $(x_1,...,x_{i-1},x_i)\in \mathcal{H}r(d) \text{ iff } (x_1,...,x_i,x_i) \in \mathcal{H}d$  
		
		 $d''\in \mathcal{H}k(d,d') \text{ iff } d'' \in \mathcal{H}d \text{ and } d'' \in \mathcal{H}d'$
		
		 $(x_1,...,x_n) \in \mathcal{H}p_0(d, d') \text {  iff } (x_1,...,x_n,d') \in \mathcal{H}d$ 
		
		$(x_1,...,x_{i-1}, y_1,..,y_n) \in \mathcal{H}p_n(d,d') \text{ iff } (x_1,...,x_{i-1}, p_0^n(d',y_1,...,y_n) )\in \mathcal{H}d$ \\

	Here we consider also the $0$-tuple as being $\emptyset$.
	Notice that elements of $D^0$ are propositions seen as objective entities according to their meaning.
	The extension functions $\mathcal{H}$ determine their Boolean values which are $T = \{\emptyset\}$ and 
	$F =\emptyset$.
	
	\begin{definition}
	 Given a model $\mathcal{M}$, an \emph{interpretation} $\mathcal{I}$ assigns $i$-ary predicate  elements of $\mathcal{D}_i$
	and $=$ to $id$. An \emph{assignment} $\mathcal{A}$ assigns to variables elements in $\mathcal{D}$. 
	\end{definition}

Bealer decomposition is crucial to define the denotation of a term relative to a model, interpretation and assignment.

\begin{definition}
	Given $\mathcal{M},\mathcal{I}$ and $\mathcal{A}$  we define the \emph{denotation} $D_{\mathcal{I},\mathcal{A},\mathcal{M}}t$
	of a term $t$ of $L^{\omega}$ as follows. If $t$ is a variable then it is $\mathcal{A}(t)$. If $t$ is an elementary term $[F(v_1,...,v_n)]_{v_1...v_n}$ then
	it is $\mathcal{I}(F)$.  Otherwise, consider the Bealer decomposition $\mathcal{B}t$ of $t$. If $\mathcal{B}t = Bt'$ (respectively $Bt't''$) where $B$ is a Bealer operation then we define
	inductively\\
	
	$D_{\mathcal{I},\mathcal{A},\mathcal{M}}t = bD_{\mathcal{I},\mathcal{A},\mathcal{M}}t'$ (respectively $D_{\mathcal{I},\mathcal{A},\mathcal{M}}t = bD_{\mathcal{I},\mathcal{A},\mathcal{M}}t' D_{\mathcal{I},\mathcal{A},\mathcal{M}}t'' )$\\

	where $b$ is $i,c,e,u,n,r,k,p_n$ if
	$B$ is $I,C,E,U,N,R,K,P_n$ respectively.
	
\end{definition}

\begin{definition}
	We say that a formula $A$ is \emph{true} for $\mathcal{M},\mathcal{I}$ and $\mathcal{A}$  (we also write $T_{\mathcal{I},\mathcal{A},\mathcal{M} }  (A)$ ) if $\mathcal{G}D_{\mathcal{I},\mathcal{A},\mathcal{M}}[A] = T$.
\end{definition}

It is simple to see that the denotation is invariant modulo renaming bound variables.
From now one, we will in most cases drop the subscripts in $D_{\mathcal{I},\mathcal{A},\mathcal{M}}$ and assume we a working with a given model , interpretation and assignment.

We are now interested in special classes of models for which T1 and T2 are sound and complete.

\begin{definition}

A model $\mathcal{M}$ is called type 1 if for all $i \geq -1$, $x,y \in \mathcal{D}_i$ 

\[\forall \mathcal{H}\in H \quad \mathcal{H}(x) = \mathcal{H}(x) \rightarrow  x = y \tag{type 1 model condition}\]

 A model $\mathcal{M}$ is called type 2 if  the operations are:
 
 \begin{enumerate}
  \item  one-to-one 
  \item disjoint on their ranges
  \item  and non-cycling

  \end{enumerate}

By the first two conditions each element of the model has a unique (possibly infinite) decomposition in terms of the operations. This decomposition is a tree having as leaves indecomposable elements. The non-cycling condition means that the same element cannot appear in more than one place in the tree.

A formula $A$ is T1-valid (we write $\vdash_{T1} A$) iff  $A$ is true for all type 1 model $\mathcal{M}$, interpretation $\mathcal{I}$ and assignment $\mathcal{A}$. In the same way a formula $A$ is T2-valid ($\vdash_{T2} A$) iff  $A$ is true for all type 2 model $\mathcal{M}$, interpretation $\mathcal{I}$ and assignment $\mathcal{A}$.

\end{definition}
	
Note that for type 2 models only $\mathcal{G}$ is relevant. The type 1 condition restricts the possible operations on the model.
	
\begin{example} A standard model for Kelley-Morse set theory furnishes an example
	of a type 1 model with a single extensional function. Type 1 models can be seen as generalised set theory models with a plurality of membership predicates $\in_{\mathcal{H}}$
	such that the axiom of extensionality is only valid globally.
	Consider a language with a single unary predicate $M$. We can construct a type 2 model with
	no constants, $D_0$, $D_1$ and $D_i = \emptyset$ for $i\geq 2$.
\end{example}

\begin{lemma}
	
	Let $v$ be an externally quantifiable variable in $[B(v)]_{\bar{x}}$ and let $t$ be free for $v$ in $[B(v)]_{\bar{x}}$. Consider
	any model structure $\mathcal{M}$ and any interpretation $\mathcal{I}$ and assignment $\mathcal{A}$. Let $\mathcal{A}'$ be an assignment which is just
	like $\mathcal{A}$ except that $\mathcal{A}'(v) = Dt$. Then
	
	\[ D_{\mathcal{A'}}[B(v)]_{\bar{x}}= D[B(t)]_{\bar{x}}     \]
\end{lemma}

\begin{proof}

	By induction on the Bealer decomposition of $[B(v)]_{\bar{x}}$. 
	The base case is either an elementary abstract (in which case there are no externally quantifiable variables) or
	else a variable. If it is not $v$ we are done and if it is $v$ the result is trivial.
	
	Suppose that $[B(v)]_{\bar{x}} = Bs(v)$ were $B$ is some unary operation. If $B$ is  $I$,  $C$, $N$, $U$ or $E$ then $t$ is still free
	for $v$ in $d(v)$ and hence by the induction hypothesis $D_{\mathcal{A'}}s(v) = Ds(t)$ so that $D_{\mathcal{A'}}[B(v)]_{\bar{x}} =
	bD_{\mathcal{A'}}s(v)  = b Ds(t) = DBs(t)$ where $b$ is the corresponding model function.
	If $B$ is $R$ then there is a problem that $s$ has a new bound variable which may occur in $t$ in such a way that  $t$ is no longer free
	for $v$ in $s$. But we may rename this bound variable to obtain an $\alpha$-equivalent term $s'$ where this problem does not occur.
	Note that the substitution \emph{does not effect which of the seven syntactic categories an intensional abstract belongs to}.
	Suppose $[B(v)]_{\bar{x}} = Bs(v)s'(v)$ for a binary operation $B$. If $B$ is $A$ then $t$ remains free for $v$ in $s(v)$ and $s(v')$
	and the previous argument applies. If $B$ is $P_k$ then we have the problem of the new bound variable in the first
	argument possibly occurring in $t$. We thus need to use $\alpha$-equivalence
	so $t$ remains free for $v$ in $s(v)$. 
	In the case of $0$-predication we have\\

	$D_{\mathcal{A'}}[B(v)]_{\bar{x}} = p_0( D_{\mathcal{A'}}[B'(v)]_{\bar{x} w}, D_{\mathcal{A'}}s'(v))$\\
	
	If $s'$ is not a variable or a variable distinct from $v$,  then we may apply the induction hypothesis to obtain:\\
	
	$p_0( D_{\mathcal{A'}}[B'(v)]_{\bar{x} w}, D_{\mathcal{A'}}s'(v))= p_0( D[B'(t)]_{\bar{x} w}, Ds'(t)) = D[B(t)]_{\bar{x}}$\\
	
	If $s'$ is $v$ then we obtain
	
	\[p_0( D_{\mathcal{A'}}[B'(v)]_{\bar{x} w}, D_{\mathcal{A'}}v) \]
	\[ = p_0( D_{\mathcal{A'}}[B'(v)]_{\bar{x} w}, \mathcal{A'}v) \]
	\[ = p_0( D [B'(t)]_{\bar{x} w}, Dt   )\]
	\[ = D[B(t)]_{\bar{x}} \]
	
	In the $k$-predication case for $k>0$ we have, for $s'(v) = [C(v)]_{\bar{y}\bar{z}}$ where $\bar{z}$ are
	relativized variables of $[C(v)]_{\bar{y}}$ in $[B(v)]_{\bar{x}}$\\
	
	$D_{\mathcal{A'}}[B(v)]_{\bar{x}}  = p_k( D_{\mathcal{A'}}[B'(v)]_{\bar{x'} w}, D_{\mathcal{A'}}[C(v)]_{\bar{y}\bar{z}} )$\\
	
	where $\bar{x'}$ is obtained from $\bar{x}$ by omitting $\bar{z}$.  The result now follows easily from induction.
	
\end{proof}

\begin{lemma}
	
	For all $\mathcal{M}$ and $\mathcal{I}$, $\mathcal{A}$ we have $D[A]_{x_1...x_k} \in \mathcal{D}_k$

\end{lemma}

\begin{proof}
	 This follows
	by easy induction from the definition of $D$ and the model operations.
	
\end{proof}

Given a sequence of Bealer operations of the form $\Sigma$  we denote its associated $n$-permutation by $\sigma$. This permutation can also be applied to $n$-tuples $(x_1,...x_n)$ and we denote such an application by $\sigma(x_1,...,x_n)$ just as we denote its application to variable sequences by $\sigma x_1....x_n$.

We will make frequent use of the following:

\begin{lemma}For any permutation $\sigma$ of $(x_1,...,x_n)$ we have
	\[ \sigma(v_1,...,v_n) \in \mathcal{H}D[F]_{\sigma x_1,...,x_n } \leftrightarrow (v_1,...,v_n) \in \mathcal{H}D[F]_{x_1...x_n} \]
	
\end{lemma}

\begin{proof}
	
Let the Bealer decomposition of  $[F]_{ x_1...x_n}$ be $\Sigma_1 \mathcal{B}[F]_{ x'_1...x'_n}$ where  $\mathcal{B}[F]_{  x'_1...x'_n}$ does not begin with $C$ or $I$. Then Bealer decomposition  of $[F]_{ \sigma x_1,...,x_n}$ will be of the form $\Sigma_2 \mathcal{B}[F]_{  x'_1...x'_n}$.
By the definition of denotation and model we have for any $(x_1,...,x_n)$:

\[ (x_1,...,x_n) \in \mathcal{H}D[F]_{ x_1...x_n} \text{ iff }  \sigma_1(x_1,...,x_n) \in  \mathcal{H}D[F]_{ x'_1...x'_n} \]

\[ (x_1,...,x_n) \in \mathcal{H}D[F]_{\sigma  x_1,...,x_n} \text{ iff }  \sigma_2(x_1,...,x_n) \in  \mathcal{H}D[F]_{ x'_1...x'_n} \]

where $ x'_1...x'_n = \sigma_1 (x_1,...,x_n) $ and $x'_1...x'_n = \sigma_2\sigma (x_1,...,x_n)$ and hence $\sigma = \sigma^{-1}_2\sigma_1$
and $\sigma_1 = \sigma_2\sigma$.

Hence given $(x_1,...x_n)$ we have $(x_1,...,x_n) \in \mathcal{H}D[F]_{ x_1...x_n} $ iff  $ \sigma_1(x_1,...,x_n) =  \sigma_2\sigma(x_1,...,x_n) \in  \mathcal{H}D[F]_{ x'_1...x'_n}$ 
iff $\sigma(x_1,...,x_n) \in  \mathcal{H}D[F]_{\sigma  x_1,...,x_n}$.

\end{proof}

	\begin{lemma} For all $\mathcal{I}$,$\mathcal{A}$,$\mathcal{M}$,
		$F(t_1,...,t_n)$ is true iff $(Dt_1,
		...,Dt_n)\in \mathcal{I}(F)$.
	\end{lemma}
	
	\begin{proof}
		By definition $F(t_1,...,t_n)$ is true iff $\mathcal{G}D([F(t_1,...,t_n)])= T$.
		But 
		
		\[[F(t_1,...,t_n)] = P_0[F(v_1,...,t_n)]_{v_1}t_1 = P_0P_p[F(v_1,v_2,...,t_n)]_{v_1v_2}t_2t_1 =P^n_0[F(v_1,...,v_n)]_{v_1...v_n}t_n...t_2t_1\]

		hence by definition $\mathcal{G}D([F(t_1,...,t_n)])= T$ iff 
		
		\[ Dt_1 \in 
		\mathcal{G}p_0(p_0(...p_0(D[F(v_1,...,v_n)]_{v_1...v_n},
		Dt_n),..., Dt_3 ), 
		Dt_2)\]
		
		iff
		
		\[ (Dt_1, Dt_2) \in 
		\mathcal{G}p_0(p_0(...p_0(D[F(v_1,...,v_n)]_{v_1...v_n},
		Dt_n)..., Dt_3 )\]
		
		and so on until obtaining
		
		\[(Dt_1,
		...,Dt_n)\in D[F(v_1,...,v_n)]_{v_1...v_n}=
		\mathcal{I}(F)   \]
		
	\end{proof}
	
	\begin{lemma}
		$T_{\mathcal{I}\mathcal{A}\mathcal{M}}(A \& B)$ iff $T_{\mathcal{I}\mathcal{A}\mathcal{M}}(A)$ and $T_{\mathcal{I}\mathcal{A}\mathcal{M}}(B)$.
		Also $T_{\mathcal{I}\mathcal{A}\mathcal{M}}(\neg A)$ iff it is not the case that $T_{\mathcal{I}\mathcal{A}\mathcal{M}}(A)$.
		
	\end{lemma}
	
	This is an immediate consequence of the definition of $T_{\mathcal{I}\mathcal{A}\mathcal{M}}$.

\section{Soundness of T1}
In this section we work in a type 1 model $\mathcal{M}$. We will use Polish
notation for the model operations and omit parentheses when possible.
We note first that Lemma 4.9  (for the case of equality) and Lemma 4.10 yield immediately:

\begin{lemma}
	For all models we have propositional tautologies, (MP), (Id)  and (L') are sound.
\end{lemma}

It is immediate by lemma 4.7 that axiom 6 is sound and the soundness of axiom 7 is obvious.

The following is of central importance
\begin{lemma} (Bealer's lemma)
	Let $v$ be free in $[A(v)]_{\bar{x}}$. Let $\mathcal{M}$ be type 1 and $\mathcal{I}$, $\mathcal{A}$ be an interpretation and assignment. Then
	
	\[D[A(v)]_{\bar{x}} = p_0D[A(v)]_{\bar{x} v}	\mathcal{A}(v) \]
\end{lemma}

Note that when we write $(x_1,...,x_n)$ we include the case $n=0$ in which case the sequence
is taken to be the empty set $\emptyset$.

\begin{proof}

We proceed by induction on the Bealer decomposition of $[A(v)]_{\bar{x}}$.

Let $\mathcal{B}[A(v)]_{\bar{x}} = K\mathcal{B}[B(v)]_{\bar{x}}\mathcal{B}[C(v)]_{\bar{x}}$.
Then 

\[D[A(v)]_{\bar{x}} = kD[B(v)]_{\bar{x}}D[C(v)]_{\bar{x}} = kp_0D[B(v)]_{\bar{x}v}\mathcal{A}(v)p_0D[C(v)]_{\bar{x}v}\mathcal{A}(v)\]

 by the induction hypothesis. We use the type 1 condition to show that

\[ kp_0D[B(v)]_{\bar{x}v}\mathcal{A}(v)p_0D[C(v)]_{\bar{x}v}\mathcal{A}(v) = p_0D[A(v)]_{\bar{x}v}\mathcal{A}(v) \]

For any $\mathcal{H}$ let  $(x_1,...,x_n) \in \mathcal{H} kp_0D[B(v)]_{\bar{x}v}\mathcal{A}(v)p_0D[C(v)]_{\bar{x}v}\mathcal{A}(v)$. Then this is equivalent
to 
\[(x_1,...,x_n) \in \mathcal{H}p_0D[B(v)]_{\bar{x}v}\mathcal{A}(v)\text{ and }  (x_1,...,x_n) \in \mathcal{H}p_0D[C(v)]_{\bar{x}v}\mathcal{A}(v) \]

which is equivalent to
\[(x_1,...,x_n,\mathcal{A}(v)) \in \mathcal{H}D[B(v)]_{\bar{x}v}\text{ and }  (x_1,...,x_n,\mathcal{A}(v)) \in \mathcal{H}D[C(v)]_{\bar{x}v} \]

The Bealer decomposition of $[A(v)]_{\bar{x}v}$ is $\Sigma K[B(v)]_{\bar{x}'}[C(v)]_{\bar{x}'}$
were $\Sigma$ is such that $\bar{x'} = \sigma\bar{x}v$ is the normalised permutation of $\bar{x}v$ for  $[A(v)]_{\bar{x}v}$. But then the above condition is equivalent to

\[\sigma(x_1,...,x_n,\mathcal{A}(v)) \in \mathcal{H}D[B(v)]_{\bar{x}'}\text{ and }  \sigma(x_1,...,x_n,\mathcal{A}(v)) \in \mathcal{H}D[C(v)]_{\bar{x}'} \]

which is equivalent to $\sigma(x_1,...,x_n,\mathcal{A}(v)) \in \mathcal{H}kDB[(v)]_{\bar{x}'}DC[(v)]_{\bar{x}'} $. Let $s$ be the composition of operations corresponding to $\Sigma$. Then the previous condition is equivalent to
$(x_1,...,x_n,\mathcal{A}(v)) \in \mathcal{H}skDB[(v)]_{\bar{x}'}DC[(v)]_{\bar{x}'} =  \mathcal{H}D[A(v)]_{\bar{x}v}$. But this is equivalent to $(x_1,...,x_n) \in \mathcal{H}p_0D[A(v)]_{\bar{x}v}\mathcal{A}(v)$. Hence the conclusion follows from the type 1 condition.

Let  $\mathcal{B}[A(v)]_{\bar{x}} = N\mathcal{B}[A'(v)]_{\bar{x}}$ where $A(v) = \neg A'(v)$. 
 We have $ D[A(v)]_{\bar{x}} = nD[A'(v)]_{\bar{x}} = np_0D[A'(v)]_{\bar{xv}}\mathcal{A}(v)$ by the induction hypothesis. We use the type 1 condition to show that 
 
 \[np_0D[A'(v)]_{\bar{x}v}\mathcal{A}(v) = p_0D[A(v)]_{\bar{xv}}\mathcal{A}(v)  \]

 Then for any $\mathcal{H}$,  
 
 \[(x_1,...,x_n) \in \mathcal{H}np_0D[A'(v)]_{\bar{x}v}\mathcal{A}(v) \text{ iff } (x_1,...,x_n) \notin \mathcal{H} p_0D[A'(v)]_{\bar{x}v}\mathcal{A}(v) \]
 
 \[\text{iff } (x_1,...,x_n,\mathcal{A}(v)) \notin \mathcal{H}[A'(v)]_{\bar{x}v} \]

Let $\bar{x}'$ be the normalised sequence for  $[A'(v)]_{\bar{x}v}$ and $\sigma$ the associated permutation. Then the above condition holds iff

\[ \sigma(x_1,...,x_n,\mathcal{A}(v)) \notin \mathcal{H}[A'(v)]_{\bar{x}'} \text{ iff }  \sigma(x_1,...,x_n,\mathcal{A}(v)) \in  \mathcal{H}n[A'(v)]_{\bar{x}'}  \]

The Bealer decomposition of $[A(v)]_{\bar{x}v}$ is $\Sigma N [A'(v)]_{\bar{x'}}$
where $\sigma$ corresponds to $\Sigma$ so $D[A(v)]_{\bar{x}v} = snD[A'(v)]_{\bar{x}'}$
where $s$ is the composition of model operation corresponding to $\Sigma$.
Hence we have  $\sigma(x_1,...,x_n,\mathcal{A}(v)) \in  \mathcal{H}n[A'(v)]_{\bar{x}'}$ iff $(x_1,....,x_n,\mathcal{A}(v)) \in \mathcal{H}D[A(v)]_{\bar{x}v}$ iff $(x_1,...,x_n)\in \mathcal{H}p_0D[A(v)]_{\bar{x}v}\mathcal{A}(v)$ and our conclusion follows from the type 1 model condition.

Let $\mathcal{B}[A(v)]_{\bar{x}} = U[A'(v)]_{\bar{x}w}$ where $A(v) = \exists w. A'(v)$. Then
$D[A(v)]_{\bar{x}} = u D[A'(v)]_{\bar{x}w} = up_0D[A'(v)]_{\bar{x}wv}\mathcal{A}(v)$ by induction.
We show that $up_0D[A'(v)]_{\bar{x}wv}\mathcal{A}(v) = p_0[A(v)]_{\bar{x}v}\mathcal{A}(v)$.
Take a $\mathcal{H}$. Then

\[ (x_1,...,x_n) \in  \mathcal{H}up_0D[A'(v)]_{\bar{x}wv}\mathcal{A}(v) \text{ iff there is $y$ such that } 
 (x_1,...,x_n,y) \in  \mathcal{H}p_0D[A'(v)]_{\bar{x}wv}\mathcal{A}(v)
 \]

\[\text{iff there is a $y$ such that }  (x_1,...,x_n,y,\mathcal{A}(v)) \in  \mathcal{H}D[A'(v)]_{\bar{x}wv} \]

As previously let $\bar{x}' = \sigma \bar{x}v$ be the normalized sequence for $[A(v)]_{\bar{x}v}$.
The Bealer decomposition of $[A(v)]_{\bar{x}v}$ is $\Sigma U\mathcal{B}[A'(v)]_{\bar{x}'w}$
so $ p_0[A(v)]_{\bar{x}v}\mathcal{A}(v) = p_0suD[A'(v)]_{\bar{x}'w}\mathcal{A}(v)$. The above
condition is equivalent to

\[\text{there is a $y$ such that }  (x'_1,...,x'_{n+1},y) \in  \mathcal{H}D[A'(v)]_{\bar{x}'w} \]

where $\mathcal{A}(v)$ occupies the position $x'_i$ of $v$ in $\bar{x}'$. But this is equivalent to

\[ (x'_1,...,x'_{n+1}) \in  \mathcal{H}uD[A'(v)]_{\bar{x}'w} \]

which is equivalent to

\[ (x_1,...,x_n, \mathcal{A}(v)) \in  \mathcal{H}suD[A'(v)]_{\bar{x}'w} \text{ iff } (x_1,...,x_n) \in  \mathcal{H}p_0suD[A'(v)]_{\bar{x}'w}\mathcal{A}(v) = \mathcal{H}p_0D[A(v)]_{\bar{x}v}\mathcal{A}(v)  \]

and the result follows from the type 1 model condition.

Let  $\mathcal{B}[A(v)]_{\bar{x}} = E\mathcal{B}[A(v)]_{\bar{x}'}$ with $\bar{x}= \bar{x}'y$ and $y$
does not occur free in $A(v)$. We have
$D[A(v)]_{\bar{x}} = eD[A(v)]_{\bar{x}'} = ep_0[A(v)]_{\bar{x}'v} \mathcal{A}(v)$ by induction.
As previously take a $\mathcal{H}$. Then

\[(x_1,....,x_n) \in \mathcal{H}ep_0[A(v)]_{\bar{x}'v} \mathcal{A}(v) \text{ iff } (x_1,....,x_{n-1}) \in \mathcal{H}p_0[A(v)]_{\bar{x}'v} \mathcal{A}(v) \]

\[\text{iff }  (x_1,....,x_{n-1},\mathcal{A}(v) ) \in \mathcal{H}[A(v)]_{\bar{x}'v}   \]

The Bealer decomposition of $[A(v)]_{\bar{x}v}$ is $\Sigma \mathcal{B}[A(v)]_{\bar{z}} = \Sigma E\mathcal{B}[A(v)]_{\bar{z'}} $ where $\bar{z} = \sigma\bar{x}v$ is the normalizing permutation which puts $v$ in its proper place and $\bar{z}'$ drops the last element of $\bar{z}$ (not $v$ and the same as $x_n$). Put $\bar{z}' = \sigma'\bar{x'}v$ and let $\sigma'$ correspond to composition of model operations $s'$. Then  the above condition is equivalent to 

\[ (x_1,....,x_{n-1},\mathcal{A}(v) ) \in \mathcal{H}s'D[A(v)]_{\sigma'\bar{x}'v} =   \mathcal{H}s'D[A(v)]_{\bar{z}'}  \]

which is equivalent to

\[ (x_1,....,x_{n-1},x_n, \mathcal{A}(v)) \in \mathcal{H}ies'D[A(v)]_{\bar{z}'} =  \mathcal{H}seD[A(v)]_{\bar{z}'} \]

which is turn is equivalent to:

\[ (x_1,....,,x_n)  \in \mathcal{H}p_0seD[A(v)]_{\bar{z}'}\mathcal{A}(v) = \mathcal{H}p_0D[A(v)]_{\bar{x}v}\mathcal{A}(v)\]

and the result follows as in the other cases.

Let $\mathcal{B}[A(v)]_{\bar{x}} = B\mathcal{B}[A(v)]_{\bar{x'}}$ with $B$ equal to $C$ or $I$ where $
\bar{x'} = \sigma_b\bar{x}$ for $b$ equal to $c$ or $i$. Then $D[A(v)]_{\bar{x}} = bD[A(v)]_{\bar{x'}}=
bp_0[A(v)]_{\bar{x'}v}\mathcal{A}(v)$ using the induction hypothesis.
We use the type 1 condition to show that

\[bp_0[A(v)]_{\bar{x'}v}\mathcal{A}(v) =  D p_0[A(v)]_{\bar{x}v}\mathcal{A}(v)\]

For any $\mathcal{H}$ we have

\[(x_1,...,x_n) \in \mathcal{H}bp_0[A(v)]_{\bar{x'}v}\mathcal{A}(v) \text{ iff } (x'_1,...,x'_n) \in \mathcal{H}p_0[A(v)]_{\bar{x'}v}\mathcal{A}(v)\]

 where $(x_1,...,x_n) = \sigma_b(x_1,...,x_n )$. But

 \[ (x'_1,...,x'_n) \in \mathcal{H}p_0[A(v)]_{\bar{x'}v}\mathcal{A}(v) \text  { iff }
 (x'_1,...,x'_n,\mathcal{A}(v)) \in \mathcal{H}[A(v)]_{\bar{x'}v} \]

 \[ \text {iff }
 (x_1,...,x_n,\mathcal{A}(v)) \in \mathcal{H}[A(v)]_{\bar{x}v}\text{ iff } (x_1,...,x_n) \in \mathcal{H}p_0[A(v)]_{\bar{x}v}\mathcal{A}(v) \]

 Hence by the type 1 model condition we get that $bp_0[A(v)]_{\bar{x'}v}\mathcal{A}(v) =  D p_0[A(v)]_{\bar{x}v}\mathcal{A}(v)$.
 
 Now let $\mathcal{B}[A(v)]_{\bar{x}} = R\mathcal{B}[A'(v)]_{\bar{x}w}$. Then $[A(v)]_{\bar{x}}$ has
 a reflected variable and hence so does $[A(v)]_{\bar{x}v}$ (it is also belongs to category 4 and is a prime reflection term). We distinguish between two cases. Either $v$ is the right-most reflected variable or it is not. Assume that is is not. Let $t$ be the right-most argument in which $t$ occurs. Then the right-most reflected variable is the last element $x_n$   of $\bar{x}$.
 We  have $D[A(v)]_{\bar{x}} = rp_0D[A'(v)]_{\bar{x}wv}\mathcal{A}(v)$ by induction, were
 $A'(v)$ is obtained from $A(v)$ by replacing the argument $t$ with $t[ w/x_n]$.
 Fix
 $\mathcal{H}$.  Then
 
 \[ (x_1,...x_n) \in \mathcal{H}rp_0D[A'(v)]_{\bar{x}wv}\mathcal{A}(v) \text{ iff }   (x_1,....,x_n,x_n,\mathcal{A}(v) ) \in \mathcal{H}D['A(v)]_{\bar{x}wv}  \]
 
 Now the Bealer decomposition of $[A(v)]_{\bar{x}v}$ in this case will be $\Sigma R\mathcal{B}[A'(v)]_{\bar{x}'w}$. Here $\bar{x}v = \sigma\bar{x}'$. $\bar{x}'$ moves $v$ to its proper place so that $\bar{x}'$ is the normalized permutation of $\bar{x}v$ except that $x_n$ is placed at the end, being the reflected variable with right-most occurrence.
 We have $D[A(v)]_{\bar{x}v} = srD[A'(v)]_{\bar{x}'w}$. Then the condition above is equivalent to:
 
 \[ (x'_1,....,x'_{n+1}, x'_{n+1} ) \in \mathcal{H}D[A'(v)]_{\bar{x}'w}  \]
with $x'_{n+1} = x_n$. But this is equivalent to:
 \[ (x'_1,....,x'_{n+1}) \in \mathcal{H}rD[A'(v)]_{\bar{x}'w}  \]
 which is turn is equivalent to 
 
  \[ (x_1,....,x_n, \mathcal{A}(v)) \in \mathcal{H}srD[A'(v)]_{\bar{x}'w}  \]

which is finally equivalent to

 \[ (x_1,....,x_n) \in \mathcal{H}p_0srD[A'(v)]_{\bar{x}'w}\mathcal{A}(v) = \mathcal{H}p_0D[A(v)]_{\bar{x}v }\mathcal{A}(v)  \]
 and the result follows as in the other cases.
 Now for the second case: $v$ is a reflected variable with the right-most occurrence. Let
 the right-most argument in which it occurs be $t$.
 The Bealer decomposition of $[A(v)]_{\bar{x}v}$ in this case will be $\Sigma R\mathcal{B}[A'(v)]_{\bar{x}'vw}$ where $A'(v)$ is obtained from $A(v)$ by replacing $t$ with $t[w/v]$. Here $\bar{x}v = \sigma\bar{x}'v$. $\bar{x}'$ moves $x_{n}$ back to its proper place so that $\bar{x'}$ is a normalized sequence. Now either in $[A'(v)]_{\bar{x}'vw}$ $v$ is still the right-most reflected variable or else it is $x_n$. If $v$ it is still the right-most reflected variable we continue with the Bealer decomposition until  we arrive at term in which the right-most reflected variable is $x_n$.
 The decomposition will be of the form:
 
 \[ [A(v)]_{\bar{x}v} = \Sigma_1 R\Sigma_2R....\Sigma_kR\mathcal{B}[A^{(k)}(v)]_{\bar{z}w_k} \]
 
 where $\bar{z}$ (normalized except for $w_k$ being placed at the end) contains the elements of $\bar{x}$,$v$ and new variables $w_1,...,w_{k-1}$.  Here $[A^{(k)}(v)]_{\bar{z}w_k}$ is now as the previous case (modulo a permutation).


  It is easy to see that

 \[ (x_1,...x_n)  \in \mathcal{H}p_0  s_1 rs_2r....s_kRD[A^{(k)}(v)]_{\bar{z}w_k}\mathcal{A}(v)  \]
 
 iff
 
 \[ (x'_1,...,x'_{n+k+1})  \in \mathcal{H}D[A^{(k)}(v)]_{\bar{z}w_k}  \]
 
 Here $(x'_1,...,x'_{n+k+1})$ follows the sequence of $\bar{z}w_k$, having $x_i$ for $x_i$ (here we are abusing notation) but has $x'_k = \mathcal{A}(v)$ whenever $x'_k$ is a $w_i$ or $v$.
 
 Now $D[A^{(k)}(v)]_{\bar{z}w_k} = sD[A^{(k)}(v)]_{\bar{u}x_n} = srD[A^{(k)}_0 (v)]_{\bar{u}x_nw}$ where $\bar{u}$ is the
 normalized permutation of $\bar{z}w_k$ but with $x_n$ omitted. Hence
 
 \[  (x'_1,...,x'_{n+k+1})  \in \mathcal{H}D[A^{(k)}(v)]_{\bar{z}w_k} \text{ iff }
 (x''_1,...,x''_{n+k}, x_n,x_n) \in \mathcal{H}D[A^{(k)}_0(v)]_{\bar{u}x_nw}   \tag{*}\]
 
 where the $x''_i$ follows the normalized sequence but with $x_n$ omitted and $A^{(k)}_0(v)$ has
 the right-most argument $t$ in which $x_n$ occurs replaced with $t[w/x_n]$.
 Now we have by induction,
 
 \[ (x_1,...,x_n) \in \mathcal{H}D[A(v)]_{\bar{x}}  \text{ iff }  (x_1,...,x_n) \in \mathcal{H}rp_0[A_1(v)]_{\bar{x}wv}\mathcal{A}(v) \]
 
 \[ \text{ iff }  (x_1,...,x_n,x_n,\mathcal{A}(v)) \in \mathcal{H}D[A_1(v)]_{\bar{x}wv} \]

where $A_1$ has the right-most argument $t$ in which $x_n$ appears replaced by $t[ w/x_n]$.
Now $[A_1(v)]_{\bar{x}wv}$ has $v$ reflected in the same arguments and positions as $[A(v)]_{\bar{x}v}$. Hence

\[ [A_1(v)]_{\bar{x}wv} = \Sigma'_1 R\Sigma'_2R....\Sigma'_kR\mathcal{B}[A_1^{(k)}(v)]_{\bar{q}w_k} \]

Hence  \[  (x_1,...,x_n,x_n,\mathcal{A}(v)) \in \mathcal{H}D[A_1(v)]_{\bar{x}wv} \text{ iff }
(q_1,....,q_{n+k+2}) \in \mathcal{H}\mathcal{D}[A_1^{(k)}(v)]_{\bar{q}w_k} \]

Here $\bar{q}$ is the normalized permutation of $\bar{x}wv$ except that $w_k$ is moved to the end.
It has new variables $w_1,...,w_{k_1}$. $(q_1,...,q_{n+k+2})$ follows this sequence and for $v$ or $w_i$ it has $\mathcal{A}$ and for $w$ it has $x_n$. It is clear that $A_1^{(k)}(v)$ is the same as
$A^{(k)}_0(v)$. Examining (*) and noticing that applying a permutation we have

\[  (x''_1,...,x''_{n+k}, x_n,x_n) \in \mathcal{H}D[A^{(k)}_0(v)]_{\bar{u}x_nw}  \text{ iff } (q_1,....,q_{n+k+2}) \in \mathcal{H}\mathcal{D}[A_0^{(k)}(v)]_{\bar{q}w_k} \]
and the result follows.

Finally, consider the case in which $\mathcal{B}[A(v)]_{\bar{x}} = P_n\mathcal{B}[A'(v)]_{\bar{u}w}t_{\bar{y}}$. Here $t_{\bar{y}}$ means that if $t = [B]_{\bar{z}}$ then
$t_{\bar{y}} =  [B]_{\bar{z}\bar{y}}$ or else $t$ is a variable and $\bar{y}$ is the empty sequence.
Let the argument in $A(v)$ which $w$ replaced be $t$. 
Here $\bar{x} = \bar{x}'\bar{x}''\bar{y}$ is normalized except that the sequence $\bar{y}$ of variables
in $\bar{x}$ which are free in $t$ are placed at the end. All arguments before $t$ are variables in $\bar{x}$. So $v$ may occur  in $t$ and in arguments after $t$ and $v$ may be reflected. Hence
we will have (if $v$ is reflected)

\[[A(v)]_{\bar{x'}\bar{x''}\bar{y}v} = \Sigma_1R\Sigma R...R\Sigma_k \Sigma [A^{(k)}(v)]_{\bar{x}'\bar{p} \bar{y}'} = \overline{\Sigma R} \Sigma [A^{(k)}(v)]_{\bar{x'}\bar{p} \bar{y'}}  \]

Here $\bar{x'}\bar{q}$ is the normalized permutation of $\bar{x}v$ and the $k$ new variables $w_k$ but with the sequence $\bar{y}'$ of variables of $\bar{q}$ which occur  free in $t$ placed at the end. $\bar{y}'$
is either equal to $\bar{y}$ or contains in addition $v$. Notice how $\bar{x'}$ is not affected by the permutations.
For clarity, we wrote the last permutation as a composition. If $v$ is not reflected then we have
only $\Sigma$ which puts $v$ in its proper place (which is either in $\bar{y}$ or in $\bar{p}$).

We have that

\[ (x'_1,...,x'_a,x''_1,...,x''_b,y_1,...,y_c, v ) \in \mathcal{H} \overline{s r} s D[A^{(k)}(v)]_{\bar{x'}\bar{p} \bar{y}'}   \]

iff

\[ (x'_1,...,x'_a,p_1,...,p_{b+k},y'_1,...,y'_{c+j}) \in \mathcal{H}  D[A^{(k)}(v)]_{\bar{x'}\bar{p} \bar{y}'}   \]

where $j$ is either $0$ and $v$ is in $\bar{p}$ or else is $1$ and $v$ is in $\bar{y}'$.
Here the $(p_1,...,p_{b+k})$  is  permutation of the $x''_i$ and $w_i$ (and possibly $v$) corresponding to the normalized sequence of these variables and which has $v$ for $w_i$.

W.l.o.g let $v$ be in $\bar{p}$ and $t$. Fix a $\mathcal{H}$. Then

\[   (x'_1,...,x'_a,x''_1,...,x''_b,y_1,...,y_c ) \in \mathcal{H}D[A(v)]_{\bar{x}'\bar{x}''\bar{y}} =  \mathcal{H}P_nD[A'_1(v)]_{\bar{x'}\bar{x}''w }Dt_{\bar{y}}
\text{ iff }    \]
\[   (x'_1,...,x'_a,x''_1,...,x''_b, p^c(Dt_{\bar{y }},y_1,...y_1) ) \in   \mathcal{H}D[A'_1(v)]_{\bar{x}\bar{x}'w } = \mathcal{H}p_0D[A'_1(v)]_{\bar{x}'\bar{x}''wv}\mathcal{A}(v)\]
by induction. This is equivalent to
\[   (x'_1,...,x'_a,x''_1,...,x''_b, p^c(Dt_{\bar{y }},y_1,...y_1),\mathcal{A}(v)) ) \in   \mathcal{H}D[A'_1(v)]_{\bar{x}'\bar{x}''wv}\]

Let the Bealer decomposition of $[A'_1(v)]_{\bar{x}'\bar{x}''wv}$
be 
\[
\Sigma'_1R\Sigma'_2 R...R\Sigma'_{k-1} \Sigma' [A'^{(k-1)}(v)]_{\bar{x}'\bar{q}} = \overline{\Sigma' R} \Sigma' [A'^{(k)}(v)]_{\bar{x}'\bar{q} }\]

where $\bar{q}$ consists of variables of $\bar{x}''$, $w$, $v$ and $k$ new variables $w_k$.

The previous condition is equivalent to

\[   (x'_1,...,x'_a,x''_1,...,x''_b, p^c(Dt_{\bar{y }},y_1,...y_1),\mathcal{A}(v)) ) \in   \mathcal{H} \overline{s'R} s' D[A'^{(k-1)}(v)]_{\bar{x}'\bar{q} }\]

which is equivalent to
\[ (x'_1,...,x'_a, r_1,....,r_s ) \in   \mathcal{H}D[A'^{(k-1)}(v)]_{\bar{x}'\bar{q} } \tag{**}\]

where $(r_1,...,r_s)$ corresponds to the sequence $\bar{q}$ with $v$ and $w_k$ corresponding to $\mathcal{A}(v)$ and
$w$ corresponding to  $p^c(Dt_{\bar{y }},y_1,...y_1)$.

Now consider

\[   (x'_1,...,x'_a,x''_1,...,x''_b,y_1,...,y_c ) \in \mathcal{H}p_0D[A(v)]_{\bar{x}\bar{x}'\bar{y}v}\mathcal{A}(v) \]

this is equivalent to

\[   (x'_1,...,x'_a,x''_1,...,x''_b,y_1,...,y_c ,\mathcal{A}(v)) \in \mathcal{H}D[A(v)]_{\bar{x}\bar{x}'\bar{y}v} = \mathcal{H} \overline{s r} s D[A^{(k)}(v)]_{\bar{x'}\bar{p} \bar{y'}}   \]

where $\bar{y'} = \bar{y''}v\bar{y'''}$.
This is turn is equivalent to

\[   (x'_1,...,x'_a,x''_1,...,x''_b,y_1,...,y_c ,\mathcal{A}(v)) \in \mathcal{H} \overline{s r} s P_{n+1}D[A^{(k)'}(v)]_{\bar{x'}\bar{p} w}t_{\bar{y'}}   \]

that is,

\[   (x'_1,...,x'_a,e_1,...,e_{b+k} ,y'_1,...,y'_{c+1}) \in \mathcal{H} P_{n+1}D[A^{(k)'}(v)]_{\bar{x'}\bar{p} w}t_{\bar{y'}}   \]
where $(e_1,...,e_{b+k+1})$ corresponds to the sequence $\bar{p}$ and has $\mathcal{A}(v)$ corresponding to $w_k$  and $(y'_1,...,y'_{c+1})$ corresponds to $\bar{y'} = \bar{y''}v\bar{y'''}$ with
$\mathcal{A}(v)$ corresponding to $v$. This in turn is the same as

\[   (x'_1,...,x'_a,e_1,...,e_{b+k}, p^{c+1}_0(Dt_{\bar{y'}},y'_1,...,y'_{c+1}) ) \in \mathcal{H}D[A^{(k)'}(v)]_{\bar{x'}\bar{p} w}   \]

It is easy to check that by induction we have

\[ p^{c+1}_0(Dt_{\bar{y'}},y'_1,...,y'_{c+1}) ) = p_0p^c_0( Dt_{\bar{y}v},y_1,...,y_c))\mathcal{A}(v)  = p^c_0( Dt_{\bar{y}},y_1,...,y_c) \]

Inspecting (**) , noticing that $[A^{(k)'}(v)]_{\bar{x'}\bar{p} w}  = [A'^{(k)}(v)]_{\bar{x'}\bar{p} w} $ (modulo renaming bound variables) and applying a permutation we get as desired.

\end{proof}

\begin{lemma} $T_{\mathcal{I}\mathcal{A}\mathcal{M}}(\exists v. A)$ iff there
	is an assignment $\mathcal{A}'$ like $\mathcal{A}$ except perhaps for what it assigns to $v$ and
	such that $T_{\mathcal{I}\mathcal{A'}\mathcal{M}}(A(v))$.
\end{lemma}
\begin{proof}
	Assume w.l.o.g that $v$ occurs free in $A$.
	We have  $T_{\mathcal{I}\mathcal{A}\mathcal{M}}(\exists v. A)$ iff 
	$\mathcal{G}D([\exists v. A]) = T$ iff 
	$\mathcal{G} uD[A]_{v} = T$.
	The last condition is equivalent to the existence of a $x\in\mathcal{D}$ such that 
	\[ x\in \mathcal{G} D[A(v)]_{v}  \]
	Let $\mathcal{A}'$ be like $\mathcal{A}$ but with $\mathcal{A}'(v) =x$.
	We must show that $\mathcal{G}D_{\mathcal{A'}}[A(v)] = T$.
	Notice that $D_{\mathcal{A'}}[A(v)]_v = D[A(v)]_v$.
	By the proof of Bealer's lemma (note that we are not using the type 1 condition) we have that:

	\[ \mathcal{G}D_{\mathcal{A'}}[A(v)] = T \text{ iff } \mathcal{G}p_0  D[A(v)]_v\mathcal{A'}(v) = T  \]

	Hence we must show that $\mathcal{G}p_0  D[A(v)]_v \mathcal{A'}(v) = T$.  But this is equivalent to $\mathcal{A}'(v) \in  \mathcal{G} D[A(v)]_{v} $ and the result follows since $\mathcal{A}'(v) = x$. The other direction is similar.

\end{proof}	
	
Hence if  $T_{\mathcal{I}\mathcal{A}\mathcal{M}}(\forall v. A(v))$ then for any assignment
$\mathcal{A}'$ which only differs at most on $v$ we have that $T_{\mathcal{I}\mathcal{A'}\mathcal{M}}(A(v))$. Hence lemma 5.3 yields:

\begin{lemma} For any models we have (Inst)  $\vDash \forall v. A(v) \rightarrow A(t)$ where $t$ is free for $v$ in $A$.

\end{lemma}

Likewise we have

\begin{lemma}
	For any models we have (QImp)  $\vDash \forall v. (A\rightarrow B(v)) \rightarrow (A \rightarrow \forall v. B(v))$  where $v$ does not occur free in  $A$.
\end{lemma}
	
\begin{proof}
 Assume $T_{\mathcal{I}\mathcal{A}\mathcal{M}}(\forall v. A( \rightarrow B(v))$ . Then for
 all assignment $\mathcal{A}'$ differing from $\mathcal{A}$ at most on $v$ we have that
 
 \[T_{\mathcal{I}\mathcal{A}'\mathcal{M}}(A \rightarrow B(v)) \tag{*}\]
 
 We must show that $T_{\mathcal{I}\mathcal{A}\mathcal{M}}(A \rightarrow \forall v.  B(v))$.
 Assume then that $T_{\mathcal{I}\mathcal{A}\mathcal{M}}(A)$. We must show that
 $T_{\mathcal{I}\mathcal{A}\mathcal{M}}(A \rightarrow \forall v.  B(v))$.
 Take any assignment $\mathcal{A}'$ differing from $\mathcal{A}$ at most on $v$.
  Since $v$ does not occur in $A$ we have that $T_{\mathcal{I}\mathcal{A}'\mathcal{M}}(A)$.
  Hence by (*) we get that  $T_{\mathcal{I}\mathcal{A}'\mathcal{M}}(\forall v. A(v))$.
  Since this holds for any $\mathcal{A}'$  we get the conclusion by lemma 5.3.
 
\end{proof}	
	
Similarly we get the soundness of the (Gen) rule:

\begin{lemma}
	For any models if $\vDash A(v)$ then $\vDash \forall v.A(v)$.
\end{lemma}	
	
	\begin{lemma}
		For all $\mathcal{M}$ of type 1 and $\mathcal{I}$, $\mathcal{A}$ and terms $t,t'$ we have
		
		\[D([t = t]) =  D([t' = t']) \] 
	\end{lemma}
	
	\begin{proof}
		Suppose that 
		\[D([t = t]) \neq  D([t' = t']) \] 
		
		Then the type 1 condition implies that there is a $\mathcal{H}$ such that
		
		\[\mathcal{H}D([t = t]) \neq  \mathcal{H}D([t' = t']) \] 
		
		But \[ [t=t] = P_0(P_0([v=w]_{vw},t )    ,t)  \]
		Hence applying the definition of $\mathcal{H}$ yields
		
		\[ \mathcal{H}D([t = t]) = T \leftrightarrow
		D(t) =  D(t)\]
		hence we have $\mathcal{H}D([t = t]) = T$ and also analogously
		$\mathcal{H}D([t' = t']) = T$ and so we obtain a contradiction.

	\end{proof}

	\begin{lemma}[\cite{QC}A8*(a)]
		For type 1 models we have (B1) $ \vDash \square (A\leftrightarrow B) \leftrightarrow [A]=[B] $.
	\end{lemma}
	
	\begin{proof}
		We must show that $\mathcal{G}D_{[\mathcal{I}\mathcal{A}\mathcal{M}}(\square (A\leftrightarrow B) \leftrightarrow [A]=[B]] ) = T$,
		which is equivalent ot showing that 
		\[ \mathcal{G}D_{\mathcal{I}\mathcal{A}\mathcal{M}}([\square (A\leftrightarrow B)]) = T 
		\text{ iff } \mathcal{G}D_{\mathcal{I}\mathcal{A}\mathcal{M}}([[A] = [B]]) = T \]
		that is,
		\[ \mathcal{G}D_{\mathcal{I}\mathcal{A}\mathcal{M}}([[A\leftrightarrow B] = [[A=B]=[A = B] ]] ) = T 
		\text{ iff } \mathcal{G}D_{\mathcal{I}\mathcal{A}\mathcal{M}}([[A] = [B]]) = T \]
		
		which is equivalent to
		
		\[ D_{\mathcal{I}\mathcal{A}\mathcal{M}}([A\leftrightarrow B]) =  D_{\mathcal{I}\mathcal{A}\mathcal{M}}([[A=B]=[A = B] ] )  
		\text{ iff }D_{\mathcal{I}\mathcal{A}\mathcal{M}}([[A])  =  D_{\mathcal{I}\mathcal{A}\mathcal{M}} ([B]) \]
		
	\end{proof}
	
	\begin{lemma}[\cite{QC} A8*(b)] For type 1 models we have (B2) $\vDash \forall v. [A(v)]_{\alpha} = [B(v)]_{\alpha} \leftrightarrow [A(v)]_{\alpha v} = [B(v)]_{\alpha v}$.
	\end{lemma}
	
	\begin{proof}
		We have
		\[D_{\mathcal{I}\mathcal{A}\mathcal{M}}[A(v)]_{\alpha} = p_0(D_{\mathcal{I}\mathcal{A}\mathcal{M}}[A(v)]_{\alpha v},
		\mathcal{A}(v)) \]
		and
		\[D_{\mathcal{I}\mathcal{A}\mathcal{M}} [B(v)]_{\alpha} = p_0(D_{\mathcal{I}\mathcal{A}\mathcal{M}}[B(v)]_{\alpha v},
		\mathcal{A}(v)) \]
		
		Consider all assignments $\mathcal{A'}$ like $\mathcal{A}$ except for the assignment on $v$. Then $D_{\mathcal{I}\mathcal{A}\mathcal{M}}[A(v)]_{\alpha v}
		=D_{\mathcal{I}\mathcal{A'}\mathcal{M}}[A(v)]_{\alpha v}$ and $D_{\mathcal{I}\mathcal{A}\mathcal{M}}[B(v)]_{\alpha v}
		=D_{\mathcal{I}\mathcal{A'}\mathcal{M}}[B(v)]_{\alpha v}$. So if $D_{\mathcal{I}\mathcal{A}\mathcal{M}}[A(v)]_{\alpha v}= 
		D_{\mathcal{I}\mathcal{A}\mathcal{M}}[B(v)]_{\alpha v}$ the for any $\mathcal{A'}$ we have 
		$D_{\mathcal{I}\mathcal{A'}\mathcal{M}}[A(v)]_{\alpha}= D_{\mathcal{I}\mathcal{A'}\mathcal{M}}[B(v)]_{\alpha}$
		and one direction follows.
		
		We now need to show that if for all such $\mathcal{A'}$ we have
		
		\[D_{\mathcal{I}\mathcal{A'}\mathcal{M}}[A(v)]_{\alpha} = 
		D_{\mathcal{I}\mathcal{A'}\mathcal{M}}[B(v)]_{\alpha}
		\]
		
		then $D_{\mathcal{I}\mathcal{A}\mathcal{M}}[A(v)]_{\alpha v}= 
		D_{\mathcal{I}\mathcal{A}\mathcal{M}}[B(v)]_{\alpha v}$.
		But if $D_{\mathcal{I}\mathcal{A}\mathcal{M}}[A(v)]_{\alpha v}\neq 
		D_{\mathcal{I}\mathcal{A}\mathcal{M}}[B(v)]_{\alpha v}$ then by the type 1 condition there would be a
		$\mathcal{H}$ and a $(x_1,...,x_m)$ such that $(x_1,....,x_m) \in \mathcal{H}D_{\mathcal{I}\mathcal{A}\mathcal{M}}[A(v)]_{\alpha v}$
		but $(x_1,....,x_m) \notin \mathcal{H}D_{\mathcal{I}\mathcal{A}\mathcal{M}}[B(v)]_{\alpha v}$.
		But taking $\mathcal{A}'(v) = x_n$ this means that 
		\[(x_1,....,x_{m-1}) \in \mathcal{H}p_0(D_{\mathcal{I}\mathcal{A}\mathcal{M}}[A(v)]_{\alpha v}, \mathcal{A'}(v)) = 
		\mathcal{H}D_{\mathcal{I}\mathcal{A'}\mathcal{M}}[A(v)]_{\alpha}\]
		
		but 
		\[(x_1,....,x_{m-1}) \notin \mathcal{H}p_0(D_{\mathcal{I}\mathcal{A}\mathcal{M}}[B(v)]_{\alpha v}, \mathcal{A'}(v))= 
		\mathcal{H}D_{\mathcal{I}\mathcal{A'}\mathcal{M}}[B(v)]_{\alpha}\]
		
		and we obtain a contradiction.
		
	\end{proof}
	
	\begin{lemma}[\cite{QC} A11*]
		In the same conditions we have (S5') $\vDash v_i \neq v_j \rightarrow \square v_i\neq v_j$.
	\end{lemma}
	
	\begin{proof}
		We must show that if $T_{\mathcal{I}\mathcal{A}\mathcal{M}}(v_i \neq v_j)$ then $T_{\mathcal{I}\mathcal{A}\mathcal{M}}( \square v_i\neq v_j )$.
		
		But if $T_{\mathcal{I}\mathcal{A}\mathcal{M}}(v_i \neq v_j)$  then $D_{\mathcal{I}\mathcal{A}\mathcal{M}}(v_i) \neq D_{\mathcal{I}\mathcal{A}\mathcal{M}}(v_k)$.
		 But this is equivalent to $\mathcal{A}(v_i) \neq \mathcal{A}(v_k)$. This means that for all $\mathcal{H}\in H$ we
		have $\mathcal{H}([v_i\neq v_j]) =T$ and the conclusion follows.
	\end{proof}
	
	\begin{lemma}[\cite{QC} A9]
		In the same conditions we have (T) $\vDash \square A \rightarrow A$ .
	
\end{lemma}
	\begin{proof}
		If $T_{\mathcal{I}\mathcal{A}\mathcal{M}}(\square A)$ then $\forall \mathcal{H}\in H$ we have
		$\mathcal{H}D_{\mathcal{I}\mathcal{A}\mathcal{M}}([A]) = T$ and so in particular $T_{\mathcal{I}\mathcal{A}\mathcal{M}}(A)$
		
	\end{proof}
	
	\begin{lemma}[\cite{QC} A10]
		In the same conditions we have (K)  $\vDash
		\square(A \rightarrow B) \rightarrow (\square A \rightarrow \square B)$.
	\end{lemma}
	
	\begin{proof}
		Assume $T_{\mathcal{I}\mathcal{A}\mathcal{M}}(\square(A \rightarrow B)$.
		Then $\forall \mathcal{H}\in H$ we have $\mathcal{H}D_{\mathcal{I}\mathcal{A}\mathcal{M}}[A \rightarrow B] = T$
		which means that if $\mathcal{H}D_{\mathcal{I}\mathcal{A}\mathcal{M}}([A])= T$ then $\mathcal{H}D_{\mathcal{I}\mathcal{A}\mathcal{M}}([B])$.
		To show that $T_{\mathcal{I}\mathcal{A}\mathcal{M}}(\square A \rightarrow \square B)$ consider any
		$\mathcal{K}\in H$. Assume that $\forall \mathcal{H} \in H$ we have $\mathcal{H}D_{\mathcal{I}\mathcal{A}\mathcal{M}}([A]) = T$.
		Then  $\mathcal{K}D_{\mathcal{I}\mathcal{A}\mathcal{M}}([A]) = T$ and hence $\mathcal{K}D_{\mathcal{I}\mathcal{A}\mathcal{M}}([B]) = T$.
		Hence we have $\forall \mathcal{H} \in H$ we have $\mathcal{H}D_{\mathcal{I}\mathcal{A}\mathcal{M}}([B]) = T$ and
		thus $T_{\mathcal{I}\mathcal{A}\mathcal{M}}(\square A \rightarrow \square B)$. 
		
	\end{proof}

	\section{Soundness of T2}
	
	We have already proven the soundness of  the standard first-order logical axioms and rules
	for any model. Also the soundess of axioms 6 and 7.  Axioms 9 to 11 are clearly sound by the type 2 model condition.

	\begin{lemma} TGen is sound:  suppose that $F$ does not occur in $A(v)$.
		If $\vdash A([F(x_1,...,x_n)]_{x_1...x_n})$ then we can infer
		$\vdash  A(t)$ whenever $t$ complex, of arity $n$ and is free for $v$ in $A$.
	\end{lemma}

	\begin{proof}
		By (L) we have
		
		\[ \vdash  [F(x_1,...,x_n)]_{x_1...x_n} = t \rightarrow A([F(x_1,...,x_n)]_{x_1...x_n}) \rightarrow A(t) \]
		
		from which it follows that
		
 	\[ \vdash A([F(x_1,...,x_n)]_{x_1...x_n})  \rightarrow   [F(x_1,...,x_n)]_{x_1...x_n} = t \rightarrow A(t) \]
		
	Assume $\vDash A([F(x_1,...,x_n)]_{x_1...x_n})$. Then by soundness of (L), (MP) and tautologies we get
	
			\[ \vDash   [F(x_1,...,x_n)]_{x_1...x_n} = t \rightarrow A(t) \]
		
		By the non-cycling condition for T2 models $F$ does not occur in $t$ and hence the interpretation of $F$ does not
		affect the denotation of $Dt$. But choose an interpretation $\mathcal{I}$ assigning to $[F(x_1,...,x_n)]_{x_1...x_n}$ the value $Dt$.
		In this case we have $\vDash [F(x_1,...,x_n)]_{x_1...x_n} = t$ and hence $T_{\mathcal{I}}(A(t))$. But since $D[A(t)]$ cannot depend on the interpretation
		is we must have $\vDash  A(t)$.

	\end{proof}
	
	\begin{lemma}
	We have that $[A]_{\bar{x}} = [B]_{\bar{x}} \rightarrow A\leftrightarrow B$ is sound.
	\end{lemma}
	\begin{proof}
		For any T2 model, interpretation and assignment we must show that $T( [A]_{\bar{x}} = [B]_{\bar{x}} \rightarrow A\leftrightarrow B)$.
		We must show that if $T( [A]_{\bar{x}} = [B]_{\bar{x}})$ then $T( A\leftrightarrow B)$.
		Now if $T( [A]_{\bar{x}} = [B]_{\bar{x}})$ then we have $D[A]_{\bar{x}} = D[B]_{\bar{x}}$.
		We now show that $T( A\leftrightarrow B)$. We show only one implication, the other is similar.
    We need only show that    $T( \forall \bar{x}. A\rightarrow B)$, that is, $\mathcal{G}D[  \forall \bar{x}. A\rightarrow B] = T$.
	Now  
	\[D[  \forall \bar{x}. A\rightarrow B] = D[  \neg \exists \bar{x}. \neg ( A\rightarrow B) ]  = D[  \neg \exists \bar{x}.  A\&  \neg B ] \]
	\[= ne.....eD[ A\&  \neg B ]_{\bar{x}}  = ue...ekD[A]_{\bar{x}}nD[B]_{\bar{x}} = ue...ekD[A]_{\bar{x}}nD[A]_{\bar{x}} \]
	
	So $\emptyset \in \mathcal{H}ue...ekD[A]_{\bar{x}}nD[A]_{\bar{x}}$ iff there is no $(x_1,....,x_n)$ such that $(x_1,...,x_n)\in \mathcal{H}D[A]_{\bar{x}}$ and $(x_1,...,x_n)\notin \mathcal{H}D[A]_{\bar{x}}$. Hence $\mathcal{G}D[  \forall \bar{x}. A\rightarrow B] = T$.
	
	\end{proof}
	Note that this rule is sound for any model.

\section{Conclusion}

\end{document}